\documentclass[letterpaper, 10 pt, conference]{ieeeconf}  

\IEEEoverridecommandlockouts                              

\usepackage{amsmath,amsfonts}
\usepackage{algorithm}
\usepackage{algpseudocode}
\usepackage{multirow}
\usepackage{graphicx}
\usepackage{cite}

\DeclareMathOperator*{\minimize}{minimize}
\DeclareMathOperator*{\argmin}{arg\,min}
\newtheorem{theorem}{Theorem}

\newtheorem{lemma}[theorem]{Lemma}

\title{\LARGE \bf
Multi-Agent Motion Planning for Simultaneous Arrival using Time-Reversed Search and Distributed Optimal Control
}

\author{Anja Hellander$^{1}$ and Daniel Axehill$^{1}$
\thanks{*This work was partially supported by the Wallenberg AI, Autonomous Systems and Software Program (WASP) funded by the Knut and Alice Wallenberg Foundation.}
\thanks{A. Hellander and D. Axehill are with the Division of Automatic Control,
        Linköping University, Sweden
        {\tt\small \{anja.hellander, daniel.axehill\}@liu.se}}%
}

\begin{document}

\maketitle
\thispagestyle{empty}
\pagestyle{empty}

\begin{abstract}

In this work we consider the multi-agent motion planning (MAMP) problem with the constraint that agents arrive at their respective goals at the same time. For the special case where all agents are initially at rest we propose a two-step method for finding optimized and kinematically feasible solutions. The first step finds an initial feasible solution by applying a state-of-the-art MAMP algorithm (conflict-based search and safe interval path planning with interval projection) backward. The algorithm is complete, and we provide necessary conditions for when it is also optimal. The second step is an improvement step where a receding-horizon optimal control problem (OCP) is posed and the solution found in the first step is used to warm-start the solver. To improve scalability we propose to solve the OCP in a distributed manner using the nonlinear alternating direction method of multipliers (NADMM).

We evaluate the proposed framework in numerical experiments on a car-like vehicle. The results show that the backward planning algorithm successfully finds feasible and collision-free solutions, and that the improvement step further improves the quality of the solutions. Compared to solving the OCPs in a centralized manner, using nonlinear ADMM reduces the computation time.

\end{abstract}

\section{INTRODUCTION}

The problem of finding feasible and collision-free trajectories for autonomous agents has been well-studied. In this paper we consider the problem of finding locally optimal feasible and collision-free trajectories for multiple agents under the constraint that all agents reach their goal simultaneously.  

\subsection{Related work}

The multi-agent pathfinding problem, where collision-free paths for multiple agents are computed, has been well-studied. Current state-of-the-art methods include the complete and optimal Conflict-Based Search (CBS) \cite{sharon2015conflict, lee2021parallel, boyarski2015icbs} and the incomplete and suboptimal Priority-Based Search (PBS) \cite{ma2019searching}. Both these methods rely on repeatedly solving single-agent pathfinding problems using, e.g., A* search \cite{hart1968formal} or safe interval path planning (SIPP) \cite{phillips2011sipp}.

More recently there has been increased interest in multi-agent motion planning (MAMP), an extension of MAPF that considers kinodynamic constraints on the agents. Examples of such work are \cite{wen2022cl, kottinger2022conflict, yan2024multi, moldagalieva2024db, su2025collaborative}. In \cite{ali2023safe} an extension of SIPP called safe interval path planning with safe intervals (SIPP-IP) is presented that considers that real-life agents cannot instantaneously transition between maximum velocity and standstill.

For some applications it may be desirable that all agents reach their goal simultaneously. This has been considered in \cite{jiang2024multi, su2019multi, yin2025collaborative}. Most work considers fixed-wing UAVs and model the agents as Dubins car. A common approach is then to compute Dubins paths and adjust the arrival time by modifying the velocity and/or the minimum turning radius \cite{zhong2023efficient,gao2014cooperative}. Similarly, \cite{shanmugavel2010co} uses Dubins paths with clothoid arcs. The problem with many of these approaches is that they first compute collision-free paths and then modify the velocity/turning radius without ascertaining that the resulting trajectories are still collision-free. This is, however, done in \cite{babel2019coordinated} which uses a single-agent planner similar to probabilistic roadmaps and performs resampling of the roadmap if it fails to construct collision-free trajectories.

As solving motion-planning problems to (global) optimality is generally intractable for non-holonomic agents in the presence of obstacles many approaches rely on first computing a candidate solution and then improving it. For single-agent planning this is done in, e.g., \cite{bergman2020optimization}. Our previous work \cite{hellander2026optimized} extended the framework in \cite{bergman2020improved} to multiple agents, but as the resulting optimization problem is solved centrally the computational effort and time required scales poorly with the number of agents.

Trajectory optimization for multiple agents can be performed in a distributed manner. For a more complete overview on distributed optimization and its applications to multi-robot systems, see \cite{yang2019survey, halsted2021survey}. In \cite{ferranti2022distributed} and \cite{tran2024parallel}, the nonlinear alternating direction method of multipliers (NADMM) \cite{themelis2020douglas} is used for distributed trajectory optimization. Like in this work, this is done in a receding-horizon fashion. 

\subsection{Contributions}

\begin{itemize}
\item We show that multi-agent motion planning problems under simultaneous arrival constraints where agents initially have zero velocity can be solved using standard MAMP algorithms by applying backwards search. Further, we proved sufficient conditions for when using an optimal MAMP algorithm for the backwards search guarantees that the solution is optimal with respect to the initial simultaneous arrival MAMP problem,
\item We propose a framework for multi-agent motion planning with simultaneous arrival consisting of two steps: a computation step where an initial feasible solution is computed followed by an improvement step  where the initial solution is improved by posing and solving an optimal control problem (OCP) in a receding-horizon manner.
\item We show that NADMM can be applied to solve the OCP in the improvement step in a distributed manner. We show using numerical examples that the computation time scales better with the number of agents than if the OCP is solved in a centralized manner.
\end{itemize}

\section{Preliminaries}

This section presents the MAMP problem with simultaneous arrival constraints as well as a state-of-the-art algorithm for MAMP.

\subsection{Problem formulation}

We consider $K$ agents operating within a workspace $W \subset \mathcal{R}^2$. The (nonlinear) dynamics of each agent $i$ are described by
\begin{equation}
\dot{x}_i(t) = f_i(x_i(t), u_i(t)), \quad x(t_0) = x_{s, i}
\end{equation}
where $x_i \in \mathcal{X}_i \subset \mathcal{R}^{n_i}$ is the state vector, $u_i \in \mathcal{U}_i \subset \mathcal{R}^{m_i}$ is the control input, $t \geq 0$ is the time elapsed, and $x_{s, i}$ is the initial state of the agent. Further, we denote the area occupied by agent $i$ while it is in state $x_i$ by $R_i(x_i) \subset \mathcal{R}^2$ and the state space of each agent by $\mathcal{X}_i = \{x_i \in \mathcal{R}^{n_i} | R_i(x_i) \subset W\}$.

The workspace also contains an obstacle region $O \in W$ that the agents may not collide with. The obstacle space for each agent is then $\mathcal{X}_{i, \text{obs}} = \{x_i \in \mathcal{X}_i | R_i(x) \cap O \neq \emptyset \}$. Agents must therefore satisfy
\begin{equation}
x_i (t) \in \mathcal{X}_{i} \setminus \mathcal{X}_{i, \text{obs}} = \mathcal{X}_{i, \text{free}},
\end{equation}

\noindent where $\mathcal{X}_{i, \text{free}}$ is the free space. Additionally, agents may not collide with each other and must therefore satisfy
\begin{equation}
R_i(x_i(t)) \cap R_j(x_j(t)) = \emptyset \quad  i \neq j.
\end{equation}

Given the initial states $x_{s, i}$ and final states $x_{f, i}$ for each agent, the multi-agent motion-planning problem consists of finding a set of feasible and collision-free trajectories $(x_i(t), u_i(t), t_{f, i})$ such that $x_i(t_0) = x_{s, i}$ and $x_i(t_{f, i}) = x_{f, i}$ for all agents $i$. The MAMP problem with simultaneous arrival requires that $t_{f, 1} = \dots = t_{f, K} = t_f$.

The optimal MAMP problem with simultaneous arrival for $k$ agents can then be posed as

\begin{subequations}
\label{eq:opt-mamp}
\begin{align}
\minimize_{\{x_i(t), u_i(t)\}_{i=1}^{K}, t_f} \quad & J = \sum_{i=1}^{k} J_i(x_i, u_i, t_f) \\
\text{s.t.} \quad \dot{x}_i(t) & = f_i(x_i(t), u_i(t)) \\
x_i(t) & \in \mathcal{X}_{i, \text{free}} \quad i = 1, \dots, K \\
u_i(t) & \in \mathcal{U}_i \\
R_i(x_i(t)) \cap R_j(x_j(t)) & = \emptyset \label{subeq:collision} \quad i \neq j \\ 
\label{subeq:init} x_i(t_0) & = x_{s, i} \\ 
\label{subeq:terminal} x_i(t_f) & = x_{f, i} 
\end{align}
\end{subequations}

\noindent where the cost functions $J_i$ are defined as
\begin{equation}
\label{eq:costl}
J_i = \int_{t_0}^{t_f} l(x_i(t), u_i(t)) \, dt
\end{equation}
\noindent where $l(x(t), u(t)) \geq 0$ is the user-defined running cost. For example, selecting $l(x, u) = 1$ results in a minimum-time problem.

\subsection{Solving the MAMP problem}

A popular MAMP algorithm is CBS. The idea behind CBS is to plan for each agent separately using a single-agent motion planner and to create a search tree where each node corresponds to a set of constraints on the trajectories. When conflicts (collisions) between agents occur new nodes are created with additional constraints to solve them. A general outline of this is shown in Algorithm \ref{alg:cbs}. First, the root node is created without any constraints and a plan consisting of trajectories for each agent is computed (line 5). Next, nodes are iteratively selected from the open list (line 8). The open list is sorted based on the cost of the corresponding solution which ensures optimality. The selected node is checked for conflicts, i.e., collisions between agents. A conflict between a pair of agents $a_i, a_j$ is selected (line 12) and two new nodes are created to which constraints on the trajectories of $a_i$ and $a_j$, respectively, are added so as to solve the conflict (line 15). These constraints are constructed from the conflict using the makeConstraint() function. Depending on implementation, constraints may specify that an agent is not allowed to be in a certain state at a certain time, or is not allowed to apply a certain motion primitive from a certain state at a certain time. The trajectories of the agents are updated so that the newly added constraints are satisfied (line 16) and the new nodes are inserted into the open list. This requires a single-agent motion planner.

A popular choice for single-agent motion planner is SIPP. SIPP uses A* search and motion primitives and is comparably efficient as it searches over time intervals rather than points in time. However, the original SIPP algorithm assumes that agents can always remain at their current state and wait. In real life many systems such as cars cannot instantaneously go from full speed to standstill and plans obtained by SIPP may therefore not be kinodynamically feasible. Instead, the extension SIPP-IP can be used which considers such kinodynamic constraints.

\begin{algorithm}[tb]
\caption{Conflict-based search (CBS)}
\label{alg:cbs}
\begin{algorithmic}[1]
\Procedure{CBS}{}
\State Input: initial state $x_{s, i}$, terminal state $x_{f, i}$ for each agent, set of motion primitives
\State Output: trajectory $(x_i(t), u_i(t))$ and final time $t_{f, i}$ for each agent
\State Root.constraints = $\emptyset$
\State Root.plan = updatePlan(Root, Root.constraints)
\State $Q$.insert(Root)
\While{not $Q$.empty()}
	\State $n = Q$.pop()
	\If{$n$.conflicts() $= 0$}
		\State \Return $n$.plan
	\EndIf
	\State $C = n$.firstConflict()
	\For{agent $i$ involved in $C$}
		\State $n' = n$
		\State $n'\text{.constraints=}$
		\mbox{$\quad n.\text{constraints} \cup \text{makeConstraint}(C, i)$}
		\State $n'$.plan = updatePlan($n'$, $n'$.constraints)
		\State $Q$.insert($n'$)
	\EndFor
\EndWhile

\EndProcedure
\end{algorithmic}
\end{algorithm}

\section{Simultaneous Arrival}\label{sec:sim-arrival}

In this section we first present an algorithm that solves the MAMP problem with simultaneous arrival constraints. Additionally, we show that the algorithm is (resolution) complete and discuss under what conditions on the cost function the algorithm is also (resolution) optimal.

\subsection{An algorithm for MAMP with simultaneous arrival constraints}

Several approaches to MAMP with simultaneous arrival follow the general idea of first computing a trajectory for each agent, fixing the arrival time to the one of the slowest agent, and finally extending the trajectories to the desired duration. The challenge with this approach is that for a general system it is not immediately straightforward how to extend a trajectory to a desired duration, in particular not while avoiding collision with other agents. There is, however, one exception: if all trajectories end with the agent at rest, each agent may safely extend its trajectory to remain at rest for any period of time without causing a collision. 

Under the assumption that all agents are at rest at their \textit{initial states}, we note that any search-based MAMP algorithm can be applied to solve MAMP with simultaneous arrival problems by searching backwards from the goal. In particular, we show how this is done for CBS using SIPP-IP as single-agent planner. This is outlined in Algorithm \ref{alg:sim-mamp}. Given the initial states $x_s$ and final states $x_f$ and a set of feasible motion primitives $\mathcal{P}$, CBS (as described in Algorithm \ref{alg:cbs}) is applied to the backward problem (line 2). In the backward problem, the initial states are $x_f$, the final states are $x_s$, and the available motion primitives are the primitives of $\mathcal{P}$ reversed in time, denoted rev($\mathcal{P}$).

Motion primitives in SIPP-IP are represented as a list of the cells swept through together with the time the cell is first swept through and the duration of the sweep. Each such instance is represented as the tuple $(dx, dy, ftt, swt, isEndCell)$, where $(dx, dy)$ represents the location of the cell relative to the state the primitive is applied in, $ftt$ is the time of first touch, $swt$ is the duration of the sweep, and $isEndCell$ indicates if the cell is still touched at the end of the primitive execution. Denoting the total relative translation of the primitive as $(\delta x, \delta y)$ and the total duration as $t_p$, the corresponding backward instance is
\begin{equation}
(dx-\delta x, dy - \delta y, t_p - ftt-swt, swt, ftt == 0).
\end{equation}

\noindent An example of this is illustrated in Figure \ref{fig:reverse}. Note that this reversal in time does not require any assumptions on the backward primitives actually being possible to execute. The backward primitives are used only for planning purposes to ensure collision-free trajectories, and the primitives in the final plan that will be executed are the original primitives.

\begin{figure}[tb]
\begin{center}
\includegraphics[clip, trim=0.5cm 6cm 1cm 2cm, width=0.45\textwidth]{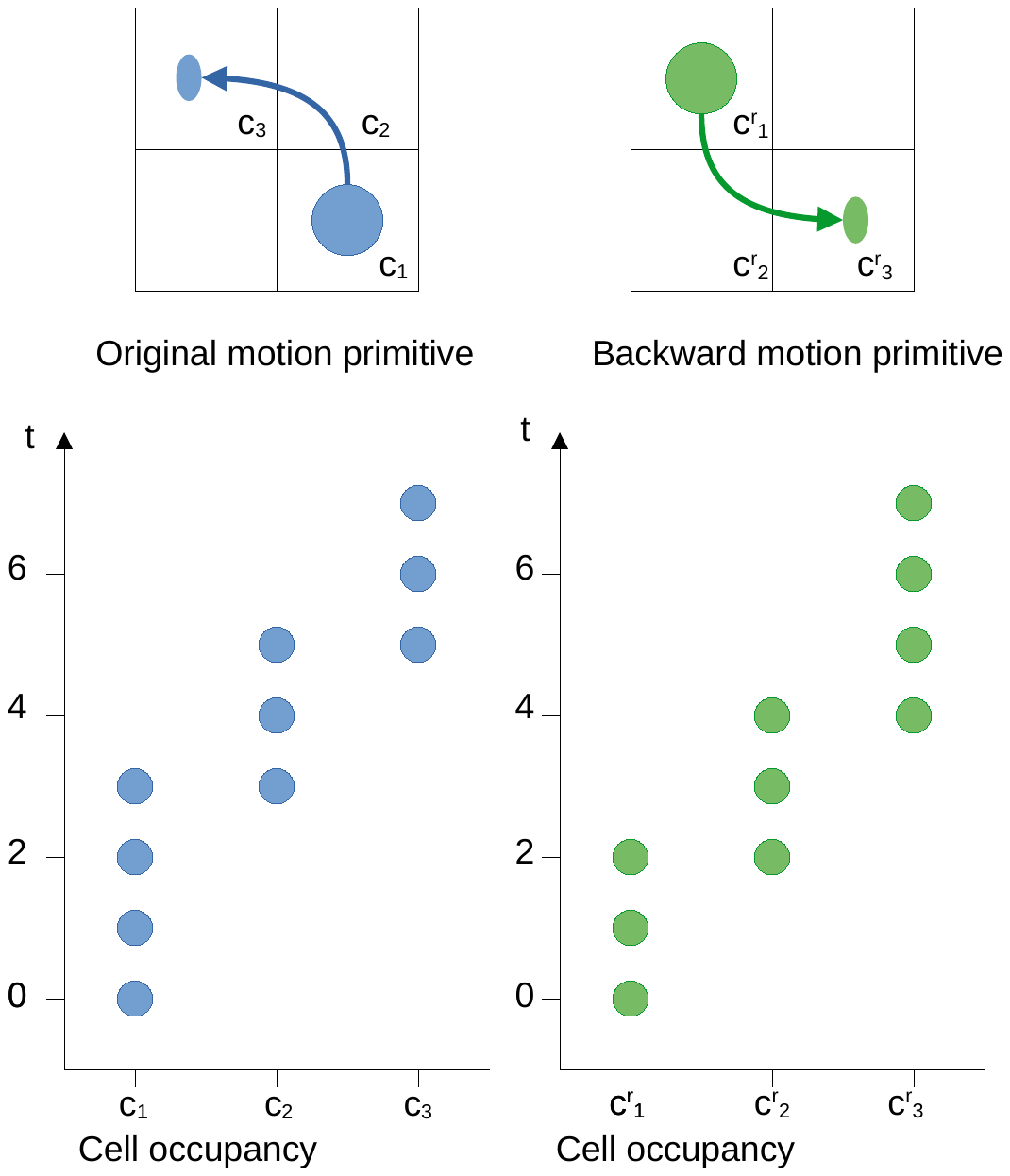} 
\caption{Illustrative example of the reversal of a motion primitive to obtain the backward primitive.} 
\label{fig:reverse}
\end{center}
\end{figure}

When the backward problem has been solved, the time required $t_f$ is then computed as the maximum duration of any individual trajectory (line 3). Finally, the trajectories are reversed in time to obtain a solution to the original problem. To ensure that all agents reach their targets simultaneously, each agent remains at its initial position for a duration of $t_f - t_{f, i}$ (line 5). This corresponds to each agent remaining at its final state in the backward problem for a duration of $t_f - t_{f, i}$, which is guaranteed to not lead to collisions.

\begin{algorithm}[tb]
\caption{MAMP with simultaneous arrival}
\label{alg:sim-mamp}
\begin{algorithmic}[1]
\State Input: initial state $x_{s, i}$ and final state $x_{f, i}$ for each agent $i$ such that $v_{i}(0) = 0$, a set of feasible motion primitives $\mathcal{P}$
\State $\{x_{\text{rev}, i}, t_{f, i}\}_{ki1}^K = $ CBS($x_f$, $x_s$, rev($\mathcal{P}$))
\State $t_{f} = \max_i t_{f, i}$
\For{$i = 1,\dots, K$}
	\State $x_i(t) = x_{\text{rev}, i}(t_{f}-t)$ if $t \geq t_f - t_{f, i}$, else $x_i(t) = x_{s, i}$
\EndFor
\State \Return $\{x_i(t)\}_{i=1}^K$
\end{algorithmic}
\end{algorithm}

\subsection{Algorithm properties}

The properties of the algorithm varies depending on the cost function used and the choice of MAMP planner on line 2 in Algorithm \ref{alg:sim-mamp} that we here assumed to be CBS with SIPP-IP.

\begin{theorem}
If the MAMP algorithm used on line 2 in Algorithm \ref{alg:sim-mamp} is complete, so is the entire Algorithm \ref{alg:sim-mamp}.
\end{theorem}

\begin{proof}
It is clear that if there exists a feasible solution to the backward MAMP problem then performing the steps on lines 4-6 in Algorithm \ref{alg:sim-mamp} results in a feasible solution to the simultaneous arrival MAMP as all agents arrive at their goals at time $t_{f}$ and all trajectories must be collision free as a collision between agents $i$ and $j$ at time $t \in [0, t_{f}]$ would indicate a collision between the same agents at time $t_{f}-t$ in the solution to the backward MAMP. Similarly, it is clear that if $x(t), t \in [0, t_f]$ is a feasible solution to the simultaneous arrival MAMP, then $x_{\text{rev}}(t) = x(t_f-t)$ is a feasible and collision free solution to the backward MAMP problem. Hence, if the algorithm used to solve the backward MAMP is complete, so is Algorithm \ref{alg:sim-mamp}.
\end{proof}

It is of interest to determine under what conditions, if any, the solution computed by Algorithm \ref{alg:sim-mamp} can be guaranteed to be optimal. It will now be shown that optimality of the MAMP algorithm used to solve the backward MAMP problem does not guarantee that the solution is an optimal solution to \eqref{eq:opt-mamp}. 

Suppose that the backward problem is solved using CBS, which is optimal. The principle behind CBS is to iteratively add collision avoidance constraints for one agent and recompute its optimal trajectory, i.e., optimize $J_i$. The optimal trajectory at each such step depends only on these constraints and not on the trajectories of the other agents. Therefore, as constraints are added the feasible set is reduced and hence it is guaranteed that the new individual trajectory will have a cost function value no lower than the previous individual trajectory and the new plan (i.e., set of trajectories for all agents) will have a cost function value no lower than the previous plan. The nodes explored by Algorithm \ref{alg:cbs} will then have non-decreasing cost function values and thus the first conflict-free plan found during the search is an optimal plan.

In the simultaneous MAMP problem, however, the constraint that all agents arrive at the same time might break this guarantee. The cost function for each agent \eqref{eq:costl} depends on $t_f$, which is the duration of the trajectories. This duration is not known when computing the initial trajectories in the backward problem, so the cost function used for agent $i$ must then be 
\begin{equation}
\tilde{J}_i = \int_{0}^{t_{f, i}} l(x_{\text{rev}, i}(t), u_{\text{rev}, i}(t)) \, dt
\end{equation}
\noindent where $t_{f, i}$ denotes the duration of the trajectory without the necessary padding to achieve the same duration of all trajectories. Note that this implies that the cost function $\tilde{J}$ minimized by the backward problem is not necessarily equal to the original cost function $J$. As additional constraints are added for an agent it is clear that $\tilde{J}_i$ cannot decrease, but for the resulting solution to be optimal it must also be guaranteed that the total cost $J$ cannot decrease. This leads us to the following results.

\begin{figure}[tb]
\begin{center}
\includegraphics[clip, trim=2cm 10cm 8cm 3cm, width=0.45\textwidth]{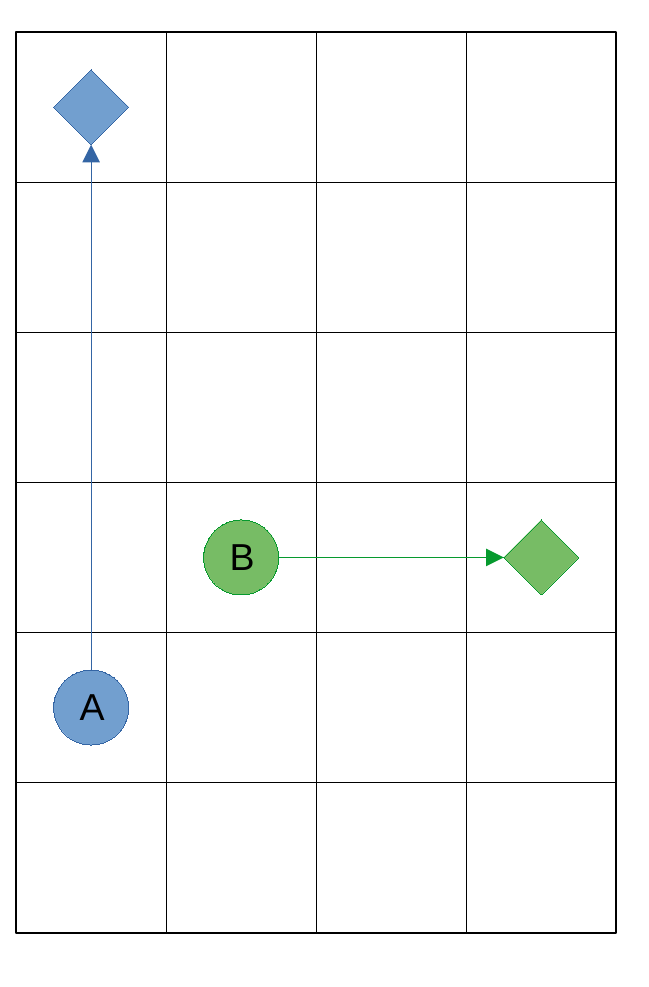} 
\caption{Example problem with two agents illustrating that the algorithm is not optimal for all choices of cost function.} 
\label{fig:example}
\end{center}
\end{figure}

\begin{lemma}
Assume an optimal algorithm is used to solve the backward MAMP problem on line 2 in Algorithm \ref{alg:sim-mamp}. Then the solution from Algorithm \ref{alg:sim-mamp} is in general not an optimal solution to \eqref{eq:opt-mamp}.
\end{lemma}

\begin{proof}
It is enough to find an example where this holds. Consider the example in Figure \ref{fig:example}. Suppose that there are two primitives available for forward motion: $p_1$ that moves two squares forward and $p_2$ that moves one square forward. Suppose further that the durations are 6 and 4 time units, respectively, that the total cost of executing the primitives are $J(p_1) = 9$ and $J(p_2) = 5$, and that there is also a primitive $p_w$ for remaining at the current square for 1 time unit that has cost $J(p_w) = 1$.
The solution obtained by finding an optimal solution to the backward problem and reversing it will be that agent A should execute $(p_1, p_1)$ and agent B $(p_w, p_w, p_w, p_w, p_w, p_w, p_1)$. This plan in reverse is an optimal solution to the backward problem, with cost $\tilde{J} = 3J(p_1) = 27$, as the cost of the wait primitives is not considered in the backward problem. However, in the original problem the cost is $J = 3J(p_1) + 6 J(p_w) = 33$. An alternative trajectory for agent B with the same arrival time is $(p_w, p_w, p_w, p_w, p_2, p_2)$. For the backward problem, such a plan would result in a total cost $\tilde{J}=2J(p_1)+2J(p_2)=28$ which is suboptimal. However, for the original problem the total cost is then $J = 2J(p_1) + 2 J(p_2) + 4 J(p_w) = 32$. This shows that although the solution the backward problem is optimal, the resulting solution to the original problem is not.
\end{proof}

\begin{theorem}
\label{th:optimality}
The solution $(x(t), u(t), t_f)$ obtained by applying Algorithm \ref{alg:sim-mamp} using an optimal algorithm to solve the backward MAMP problem  on line 2 is optimal if
\begin{equation}
\int_{0}^{t} l(x, \mathbf{0}) dt = 0
\end{equation}
\noindent for all states $x$ and all $t \in [0, t_f]$, i.e., if waiting at the initial state for time $t$ has zero cost.
\end{theorem}

\begin{proof}
We have
\begin{equation}
\begin{aligned}
J_i &= \int_{0}^{t_f} l(x_i(t), u_i(t) \, dt = \\
&= \int_{0}^{t_f-t_{f, i}} l(x_{s, i}, \mathbf{0}) \, dt + \int_{t_f-t_{f, i}}^{t_f} l(x_i(t), u_i(t)) \, dt = \\
&=  0 + \int_{0}^{t_{f, i}} l(x_{\text{rev}, i}(t), u_{\text{rev}, i}(t) \, dt = \tilde{J}_i.
\end{aligned}
\end{equation}
Suppose that $(x(t), u(t), t_f)$ is not an optimal solution to the simultaneous arrival MAMP problem. Then there exists another solution $(x^*(t), u^*(t), t_f^*)$ such that $J(x^*(t), u^*(t), t_f^*) < J(x(t), u^*(t), t_f)$. However, then there must exist a corresponding solution to the backward problem $(x^*_{\text{rev}}(t), u^*_{\text{rev}}(t), t_f^*)$ such that $\tilde{J}(x^*_{\text{rev}}(t), u^*_{\text{rev}}(t), t_f^*) = J(x^*(t), u^*(t), t_f^*) < J(x(t), u^*(t), t_f) = \tilde{J}(x_{\text{rev}}(t), u_{\text{rev}}(t), t_f)$ which contradicts that the algorithm used to solve the backward MAMP is optimal. Hence, the solution found by Algorithm \ref{alg:sim-mamp} must be optimal as well.
\end{proof}

We have assumed that $J = \sum_i= J_i$ and $\tilde{J} = \sum_i \tilde{J}_i$. It is also possible to use $\tilde{J} = \max_i \tilde{J}_i$. Doing so makes it clear that the condition in Theorem \ref{th:optimality} is sufficient but not necessary as there is another special case of cost function such that Algorithm \ref{alg:sim-mamp} is optimal. 

\begin{theorem}
If $J = ct_f$ for any constant $c > 0$, then Algorithm \ref{alg:sim-mamp} is optimal.
\end{theorem}

\begin{proof}
Select $\tilde{J} = \max_i \tilde{J}_i$ where $\tilde{J}_i = ct_{f, i}$. Then $J = \tilde{J}$ and by the same reasoning as in the proof of Theorem \ref{th:optimality} the solution found by Algorithm \ref{alg:sim-mamp} must be optimal.
\end{proof}

\section{Improvement using NADMM}

Consider again the optimal MAMP problem with simultaneous arrival \eqref{eq:opt-mamp}. Applying the approach outlined in Section \ref{sec:sim-arrival} yields a feasible and in general suboptimal solution. We propose to improve this solution, i.e., to try to compute a solution with a better objective function value. In our previous work \cite{hellander2026optimized} this was done by posing \eqref{eq:opt-mamp} as an OCP and solving it using the available feasible solution to warm-start the numerical solver. In that previous approach, a single OCP is posed, jointly optimizing the trajectories of all agents. As this approach scales poorly with the number of agents we will here propose to instead reformulate \eqref{eq:opt-mamp} as a NADMM problem, allowing it to be solved in a distributed manner. We also propose to improve the solution in a receding-horizon fashion where only parts of the trajectories are improved at a time to, e.g., reduce latency.

\subsection{NADMM}

The alternating direction method of multipliers (ADMM) is an algorithm that is well-suited for distributed convex optimization, combining aspects of dual decomposition and augmented Lagrangian methods \cite{boyd2010distributed}. A version of ADMM that can be applied to a larger range of problems is NADMM proposed by \cite{themelis2020douglas}. The NADMM considers problems on the form

\begin{subequations}
\label{eq:nadmm}
\begin{align}
\minimize_{w, v} f(w) + g(v) \\
\text{s.t.} \quad Aw + Bv = b.
\end{align}
\end{subequations}

The problem \eqref{eq:nadmm} is solved by iteratively applying the steps

\begin{equation}
\begin{aligned}
z^{+1/2} &= z - \beta (1 - \lambda) (Aw + Bv - b) \\
w^+ &= \text{arg min } \mathcal{L}_\beta (\cdot, v, z^{+1/2}) \\
z^+ &= z^{+1/2} + \beta (Aw + Bv - b) \\
v^+ &= \text{arg min } \mathcal{L}_\beta (w, \cdot, z^{+1/2}) \\
\end{aligned}
\end{equation}

\noindent where $\beta > 0$ is the penalty parameter, $\lambda \in (0, 2)$ is the relaxation parameter, $z$ is the Lagrange multiplier and 
\begin{equation}
\begin{aligned}
\mathcal{L}_\beta(w, v, z) = f(w) + g(v) &+ \langle z, Aw + Bv - b \rangle
\\ &+ \frac{\beta}{2} \Vert Aw + Bv - b \Vert^2
\end{aligned}
\end{equation}

\noindent is the augmented Lagrangian.

To formulate \eqref{eq:opt-mamp} as a distributed optimization problem we introduce local decision variables for each agent. The global decision variables (corresponding to $v$ in \eqref{eq:opt-mamp}) are $\xi = [x_1^T(t_0:t_f), \dots, x_K^T(t_0:t_f), t_f]^T$. Each agent $i$ has the local decision variables (corresponding to $w$ in \eqref{eq:opt-mamp}) $\tilde{x}_{ij}$, $j=1, \dots K$ denoting the trajectory of agent $j$. That is, in addition to deciding its own trajectory $\tilde{x}_{ii}$ agent $i$ also proposes trajectories $\tilde{x}_{ij}$, $i \neq j$, for the other agents that are collision-free with respect to $\tilde{x}_{ii}$. Agent $i$ also has local variables $\tilde{u}_i$ denoting its own control signals, and $t_{f, i}$ denoting its terminal time. 

Previous work that used NADMM for trajectory optimization \cite{ferranti2022distributed, tran2024parallel} considered a model predictive control (MPC) formulation where the terminal time of the trajectories considered is always fixed. In this work we extend the class of trajectory optimization problems that NADMM can be applied to by including the terminal time $t_f$ as a decision variable, allowing for $t_f$ to be part of the objective function to minimize.

The distributed optimization problem to be solved by each agent $i$ is then formulated as

\begin{subequations}
\label{eq:dist-opt}
\begin{align}
\minimize_{\tilde{x}_{ii}(t), \tilde{u}_{i}(t), t_{f, i}} \quad &  J_i(x_{ii}, u_{i}, t_{f, i}) \\
\text{s.t.} \quad \dot{\tilde{x}}_{ii}(t) & = f_i(\tilde{x}_{ii}(t), \tilde{u}_{i}(t)) \\
\tilde{x}_{ii}(t) & \in \mathcal{X}_{i, \text{free}} \quad i = 1, \dots, k \\
\tilde{u}_{i}(t) & \in \mathcal{U}_i \\
g(\tilde{x}_{ii}(t), \tilde{x}_{ij}(t)) & \leq 0 \\ 
 \tilde{x}_{ii}(t_0) & = x_{s, i} \\ 
 \tilde{x}_{ii}(t_{f, i}) & = x_{i, f} \\
 W_{ij}x_{ij}(t) &= W_{ij}x_j(t), \quad i \neq j \\
 W_t t_{f, i} &= W_t t_f 
\end{align}
\end{subequations}

\noindent where $ g(x_{ii}(t), x_{ij}(t)) \leq 0$ represents the collision avoidance constraint and $J_i$ is as in \eqref{eq:costl}. The last two constraints denote that the values of the local variables should equal the values of the global consensus variables. Following \cite{tran2024parallel}, we weigh these constraints using $W_{ij}$ and $W_t$. These weights do not affect the solution if solving \eqref{eq:dist-opt} directly, but as they appear in the augmented Lagrangian they affect the solution when NADMM is applied. To solve the infinite-dimensional OCP in \eqref{eq:dist-opt}, we apply a direct transcription method. We assume a zero-order hold on the control inputs $u$ and discretize the state trajectories and constraints at N time steps, resulting in a finite-dimensional nonlinear programming problem.

The resulting NADMM algorithm is outlined in Algorithm \ref{alg:admm}. For a given number of iterations $s_{\text{max}}$ each agent in parallel performs the steps on lines 5 to 11. First, they receive the global variable estimates from other agents and update the global variables (lines 5-6). Next, the Lagrangian multiplier and local variables are updated (lines 7-9) which involves minimizing the augmented Lagrangian
\begin{equation}
\mathcal{L}_i = J_i + \langle z, W(\bar{\xi}-\xi) \rangle + \frac{\beta}{2} \Vert W(\bar{\xi}-\xi)\Vert^2.
\end{equation}
\noindent Finally, a global variable estimate is computed and sent to the other agents (lines 10-11).

\begin{algorithm}[tb]
\caption{Nonlinear ADMM for simultaneous MAMP}
\label{alg:admm}
\begin{algorithmic}[1]
\State Given: $\xi^0$
\State $s = 0$
\While{$s \leq s_{\text{max}}$}
	\For{agents $i = 1, \dots, K$ in parallel}
		\State Receive $\hat{\xi}_j$ from neighbours
		\State $\xi^{s+1} = \frac{1}{K} \sum_{j=1}^K \hat{\xi}_j$
		\State $z_i^{s+1/2} = z_i^s + \beta(1-\lambda)(\tilde{\xi}_i^{s} - \xi^{s+1})$
		\State $\begin{bmatrix} \tilde{u}_i^{s+1} \\ \tilde{\xi}_i^{s+1} \end{bmatrix} = \argmin_{\xi, u} \mathcal{L}_i((\xi, u), \xi^{s+1}, z_i^{s+1/2})$
		\State $z_i^{s+1} = z_i^{s+1/2} + \beta (\tilde{\xi}_i^{s+1} - \xi^{s+1})$
		\State $\hat{\xi}_i = \tilde{\xi}_i^{s+1} + \frac{1}{\beta} z_i^{s+1}$
		\State Transmit $\hat{\xi}_i$ to neighbours
		
	\EndFor
	\State $s = s + 1$
\EndWhile
\end{algorithmic}
\end{algorithm}

\subsection{A receding-horizon improvement algorithm}

The improvement of the nominal solution is performed using an iterative receding-horizon approach. The approach is an extension of the approach for single-agent systems presented in \cite{bergman2020optimization}. The principle idea of this approach is that at each iteration $k$ at time $k\delta$, a new trajectory $(\bar{x}^{(k)}(t), \bar{u}^{(k)}(t))$ is computed by solving an OCP over the sliding time window $[k\delta, k\delta + T]$.

The receding-horizon improvement is outlined in Algorithm \ref{alg:improvement}. Inputs to the algorithm are the feasible and collision-free trajectories $\{(x_i(t), u_i(t))\}_{i=1}^{K}$, each of length $t_f$, as well as the user-defined horizon length $T$ and step length $\delta$. At each iteration $k$ the algorithm considers the trajectories during the time interval $[k\delta, k\delta + T_k]$ where $0 < T_k \leq T$. These trajectories are optimized using Algorithm \ref{alg:admm}, with $x_i(k\delta)$ and $x_i(k\delta + T)$ as the initial and terminal states, respectively, for agent $i$ (line 6). This results in the optimized trajectories $(x^*(t), u^*(t))$ of duration $T_k^*$. As in \cite{bergman2020optimization}, the horizon length $T_k$ is a decision variable. Note that the resulting trajectory may therefore have a duration that is shorter than the original duration $T$.

A new candidate solution is then constructed (line 7) by assigning the control signal

\begin{equation}
u_{i, \text{cand}}(t) = 
	\begin{cases}
		\bar{u}_i^{(k)}(t) & t \in [0, k\delta)\\
		u^*(t-k\delta) & t \in [k\delta, k\delta+T_k^*]\\
		\bar{u}_i^{(k)}(t-\Delta T) & t \in (k\delta+T_k^*, t_f-\Delta T]
	\end{cases}
\end{equation}

\noindent where $\Delta T = T - T_k^*$ and applying it to $x_i(0)$ to obtain $x_i(t)$. The time duration of the new solution is $t_{\text{cand}, f} = t_f - \Delta T$ and is the same for all agents. Note that if $\bar{x}^{(k)}$ is feasible and collision-free, and $x^*(t)$, $t \in [k\delta, k\delta+T_k^*]$ is also feasible and collision-free, then this also holds for $\bar{x}^{(k+1)}$. All agents simultaneously reach the position $\bar{x}^{(k)}(k\delta + T)$ at the same time $k\delta + T_k^*$ without collision (since $x^*(t)$ is feasible and collision-free) and then follow the trajectories $\bar{x}^{(k)}(t)$, $t \in [k\delta + T, t_f]$ displaced in time by $\Delta T$ which are also feasible and collision-free. An illustrative example of one receding-horizon iteration is shown in Figure \ref{fig:rh}.

If the candidate solution is feasible and better than the previous solution it is accepted as the new current solution (line 9), otherwise the previous solution is kept (line 11). Thus, it is ensured that the current solution is always feasible and further that
\begin{equation}
J(\bar{x}^{(k+1)}, \bar{u}^{(k+1)}) \leq J(\bar{x}^{(k)}, \bar{u}^{(k)}) \leq \dots \leq J(\bar{x}^{(0)}, \bar{u}^{(0)}).
\end{equation}

\begin{algorithm}[tb]
\caption{Improvement}
\label{alg:improvement}
\begin{algorithmic}[1]
\State Input: feasible and collision-free trajectories $\{(x_i(t), u_i(t))\}_{i=1}^{K}$ of length $t_f$, horizon length $T$, step length $\delta$
\State $t_{\text{curr}} = 0$
\State $\bar{x}_i^{(0)}(t) = x_i(t), \bar{u}_i^{(0)}(t) = u_i(t), \bar{t}_f^{(0)} = t_f$
\State $k = 0$
\While{$t_{\text{curr}} + T \leq \bar{t}_f^{(k)}$}
	\State $(x^*(t), u^*(t), T^*)$ = NADMM($(\bar{x}^{(k)}(t_{\text{curr}}:t_{\text{curr}}+T), \bar{u}^{(k)}(t_{\text{curr}}:t_{\text{curr}}+T)$)
	\State $x_{\text{cand}}, u_{\text{cand}}, t_{\text{cand}, f} =$ computeCandidate()
	\If{feasible and $J(x_{\text{cand}}, u_{\text{cand}}) \leq J(\bar{x}^{(k)}, \bar{u}^{(k)})$}
		\State $\bar{x}^{k+1}, \bar{u}^{k+1}, \bar{t}_f^{k+1} = x_{\text{cand}}, u_{\text{cand}}, t_{\text{cand}, f}$
	\Else
		\State $\bar{x}^{k+1}, \bar{u}^{k+1}, \bar{t}_f^{k+1} = \bar{x}^{k}, \bar{u}^{k}, \bar{t}_f^{k}$
	\EndIf
	\State $k = k + 1$
	\State $t_{\text{curr}} = t_{\text{curr}} + \delta$
\EndWhile
\end{algorithmic}
\end{algorithm}

\begin{figure}[tb]
\begin{center}
\includegraphics[clip, trim=3.5cm 8.5cm 3.5cm 9cm, width=0.45\textwidth]{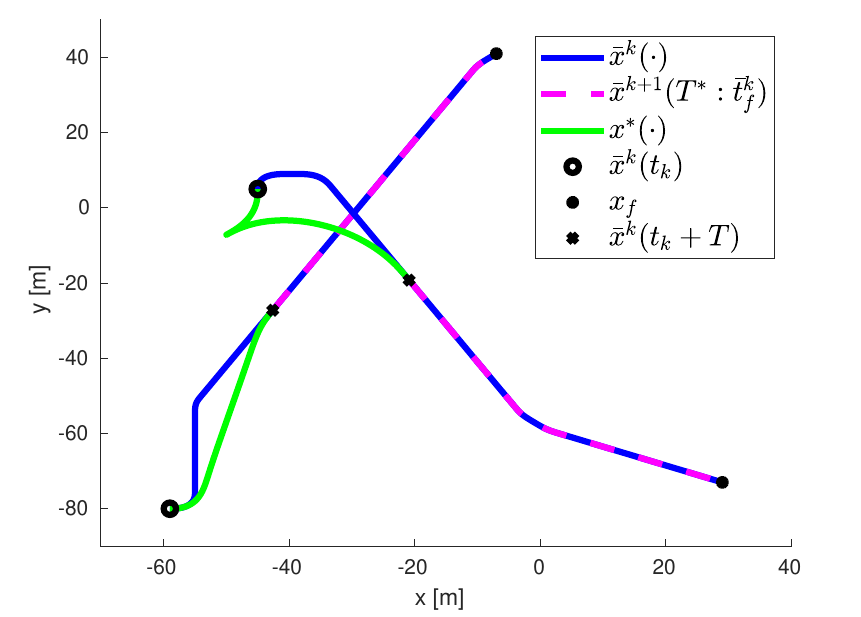} 
\caption{Illustrative example of receding-horizon improvement. Optimization over the horizon has resulted in the green trajectory. Beyond the horizon the previous solution shown in blue is used to guarantee feasibility.}
\label{fig:rh}
\end{center}
\end{figure}

\section{Numerical Experiments}

In this section the proposed planning algorithms are implemented and evaluated in problems with a car-like vehicle. 

The computation of the motion primitives and the improvement step are both implemented in Python using CasADi \cite{andersson2019casadi} with IPOPT \cite{wachter2006implementation} and the ma57 solver from HSL.

\subsection{Vehicle Model}

The car-like vehicle is modelled as a kinematic bicycle \cite{lavalle2006planning} with augmented state $\mathbf{x}(t) = [x(t), y(t), \theta(t), \alpha(t), \omega(t), v(t), a(t)]^T$. Here, $(x, y)$ denotes the position of the car and $\theta$ its orientation. Further, $\alpha$ and $\omega$ are the steering angle and steering angle rate, respectively, of the car and $v$ and $a$ are the longitudinal velocity and acceleration, respectively. The control signals to the system are $\mathbf{u} = [u_\omega, u_a]^T$. The system can then be described by
\begin{equation}
\dot{\mathbf{x}} = f(\mathbf{x}, \mathbf{u}) = 
\begin{pmatrix}
v \cos{\theta} \\
v \sin{\theta} \\
v \frac{\tan{\alpha}}{L} \\
\omega \\
u_\omega \\
a \\
u_a
\end{pmatrix}
\end{equation}

\noindent where $L$ is the wheelbase of the car.

The cost function used is 
\begin{equation}
l(\mathbf{x}, \mathbf{u}) = 1 + \frac{1}{2}(\alpha^2 + 10 \omega^2 + a^2 + \mathbf{u}^T \mathbf{u}).
\end{equation}

\noindent This is the cost function used for primitive generation, by SIPP-IP and by the continuous improvement step. 

The primitive generation is performed as described in \cite{hellander2026optimized} to generate synchronized primitives as required by SIPP-IP. A total of 1376 primitives are generated.

\subsection{Simultaneous arrival}

We randomly generate 100 problem instances for $n = 2, 3, 4, 5$ agents in an obstacle-free $200 \times 200$ map. All problem instances are generated so as to assure no collisions at the initial and terminal positions. For comparison, we implement a baseline algorithm that follows the outline in \cite{shanmugavel2010co} of (i) finding trajectories, (ii) ensuring no collisions, (iii) ensuring the same duration. To find trajectories and ensure that there are no collisions we use CBS with the original motion primitives. To ensure simultaneous arrival, all trajectories are extended to the same duration by padding with waiting at the initial positions. Note that, as in \cite{shanmugavel2010co} this may result in the trajectories no longer being collision-free.

The resulting success rate, i.e., the number of problems where a feasible and collision-free solution was found within a time limit of 100 s is shown in Figure \ref{fig:success_rate}. It can be seen that the proposed algorithm outperforms the baseline algorithm. For two agents the performance is similar, with both algorithms having a success rate of $96 \%$, but as the number of agents increases the difference between the algorithms increases. For five agents the baseline algorithm solves only $58 \%$ whereas the proposed algorithm solves $84 \%$.

The computation times required by the algorithms for the successfully solved problems are shown in Figure \ref{fig:runtime}. It can be seen that the proposed algorithm generally does not perform worse than the baseline algorithm. The only exception is for five agents where the algorithms show a similar median computation time but the proposed algorithm shows a greater variation. It can also be seen that the computation time increases with the number of agents.

\begin{figure}[tb]
\begin{center}
\includegraphics[clip, trim=4cm 8.5cm 3.5cm 9cm, width=0.45\textwidth]{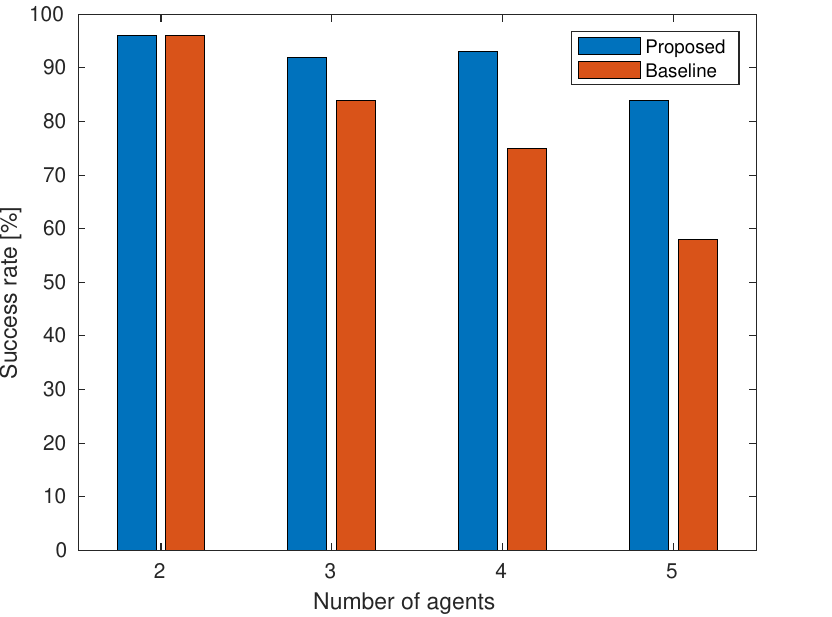} 
\caption{Success rate for the proposed and baseline algorithms.} 
\label{fig:success_rate}
\end{center}
\end{figure}

\begin{figure}[tb]
\begin{center}
\includegraphics[clip, trim=4cm 8.5cm 3.5cm 9cm, width=0.45\textwidth]{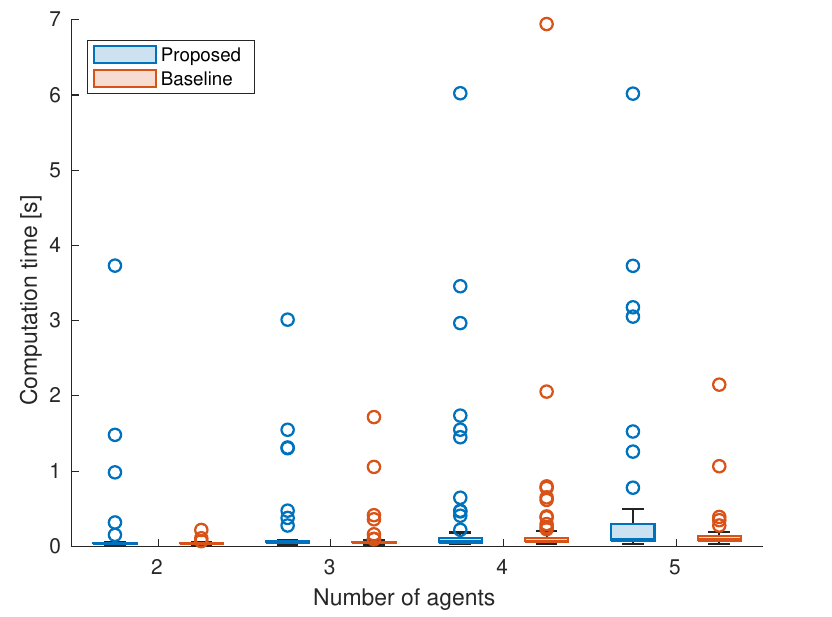} 
\caption{Distribution of computation times for the proposed and baseline algorithms.} 
\label{fig:runtime}
\end{center}
\end{figure}

\subsection{Improvement step}


%

We applied the proposed improvement step to selected problems with $n=2, 3, 4, 5$ agents. The number of iterations used by the NADMM algorithm was set to $s_{\text{max}} = 5$. For comparison we also applied the same receding-horizon algorithm but instead of solving \eqref{eq:dist-opt} as a distributed optimization problem we solve it in a centralized manner where all trajectories are optimized by a single solver. Examples of resulting trajectories are shown in Figure \ref{fig:trajectories}. The resulting cost function improvement for varying horizon lengths is shown in Figure \ref{fig:improvement}. The full duration of the original plan is just over 160 s for $n=2, 4, 5$ agents and just over 140 s for $n=3$ agents. It can be seen that in general the centralized algorithm achieves a larger improvement than the distributed algorithm using NADMM. For shorter horizon lengths ($< 60$ s) the algorithm using NADMM performs similar to the centralized, but for longer horizon lengths the gap increases. It can also be seen that as the horizon becomes longer the additional improvement decreases. Interesting to note is also that when the horizon length reaches the full duration, the improvement achieved by the algorithm using NADMM suddenly decreases drastically. This will be discussed in greater detail later in this section.

To evaluate the two algorithms we also compared their latency times, i.e., the computation time required for the first horizon. The results are shown in Figure \ref{fig:latency}. It can be seen that the general trend is that the latency increases linearly with the horizon length and the number of agents. The distributed algorithm with NADMM scales much better than the centralized algorithm. For two agents the computation times are very similar, but already for three agents the higher performance of the proposed NADMM scheme becomes clear since the time required for three agents by the centralized algorithm is similar to the time required for five agents by the NADMM algorithm. 

The improvement in mission time is shown in Figure \ref{fig:time_reduction}. The trend is similar to that of the cost function improvement in Figure \ref{fig:improvement}. The proposed algorithm that uses NADMM performs worse than the centralized algorithm and the gap increases with the horizon length. When the full horizon is reached, the algorithm using NADMM performs worse than for short horizon lengths. For, e.g., five agents the resulting time improvement is less than the latency time in Figure \ref{fig:latency}. For the centralized algorithm, using the full horizon results in a large increase in time improvement.

That the distributed algorithm performs worse when using the full horizon might seem counterintuitive. In general, a longer horizon should mean a higher degree of freedom and make it possible to achieve greater improvement as in the centralized case. However, the number of NADMM iterations is limited to $s_{\text{max}}$ and the constraints (11h)-(11i) limits the steps taken in $t_f$ in each such iteration. This limits the improvement that is possible each time NADMM is called. For shorter horizon lengths the NADMM algorithm will be called many times until the condition on line 5 in Algorithm \ref{alg:improvement} does not hold anymore so that even if the improvement each time the algorithm is executed is small it will build up over the trajectory since the improved trajectory from the last sample is used as a start. When the horizon $T$ is very close to the full duration the NADMM algorithm will only be called a few times as the condition on line 5 in Algorithm \ref{alg:improvement} stops holding after only a few iterations of the receding-horizon algorithm. This limits the improvement that can be achieved. This can be seen in Figure \ref{fig:iterations}, where the improvement achieved on the problem with five agents and horizon length 160 s for varying number of NADMM iterations is shown. The improvement increases roughly linearly with the number of iterations up until around 65 iterations when it reaches a plateau and achieves similar improvement as the centralized algorithm. It can also be seen that the latency time increases linearly with the number of iterations and that it is always higher than the mission time improvement with the proof-of-concept implementation and hardware available for the experiments.

\begin{figure}[tb]
\begin{center}
\includegraphics[clip, trim=4cm 8.5cm 3.5cm 9cm, width=0.45\textwidth]{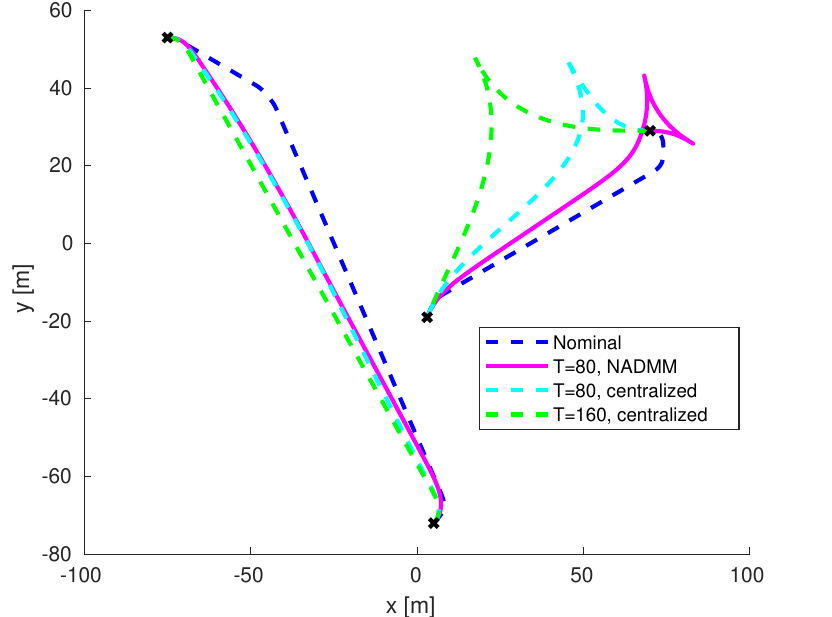} 
\caption{Examples of resulting trajectories after the improvement step.} 
\label{fig:trajectories}
\end{center}
\end{figure}

\begin{figure}[tb]
\begin{center}
\includegraphics[clip, trim=4cm 8.5cm 3.5cm 9cm, width=0.45\textwidth]{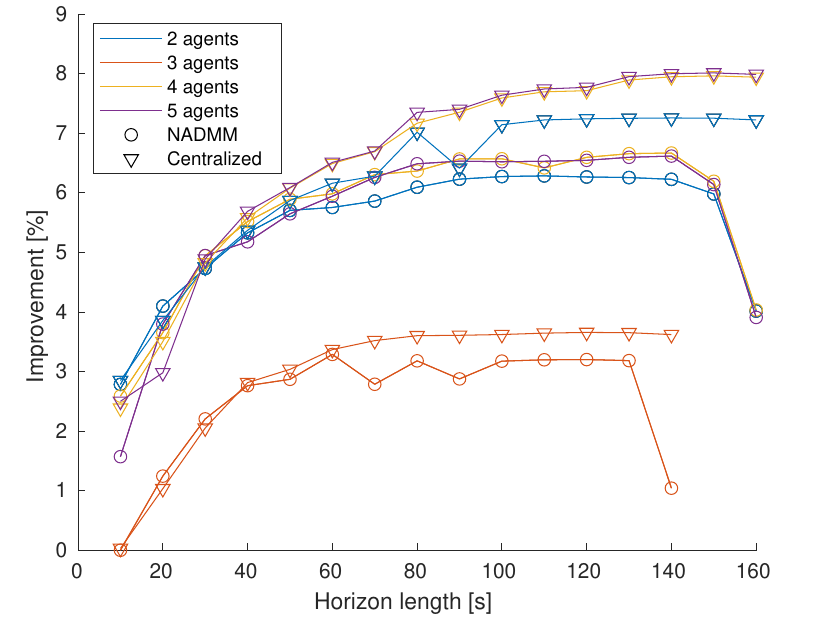} 
\caption{Cost function improvement by the proposed receding-horizon algorithm using NADMM as well as a centralized receding-horizon algorithm.} 
\label{fig:improvement}
\end{center}
\end{figure}

\begin{figure}[tb]
\begin{center}
\includegraphics[clip, trim=4cm 8.5cm 3.5cm 9cm, width=0.45\textwidth]{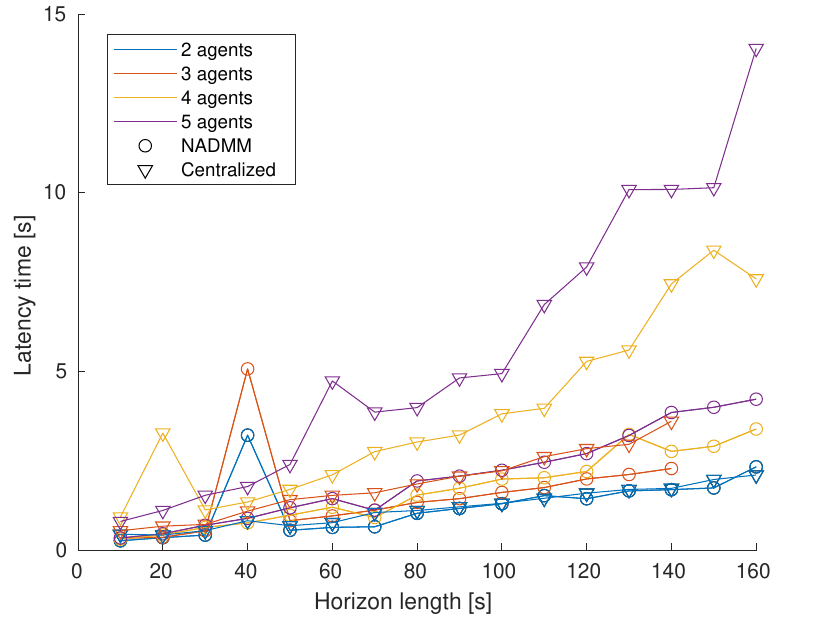} 
\caption{The latency time, i.e., the computation time required for the first horizon.} 
\label{fig:latency}
\end{center}
\end{figure}

\begin{figure}[tb]
\begin{center}
\includegraphics[clip, trim=4cm 8.5cm 3.5cm 9cm, width=0.45\textwidth]{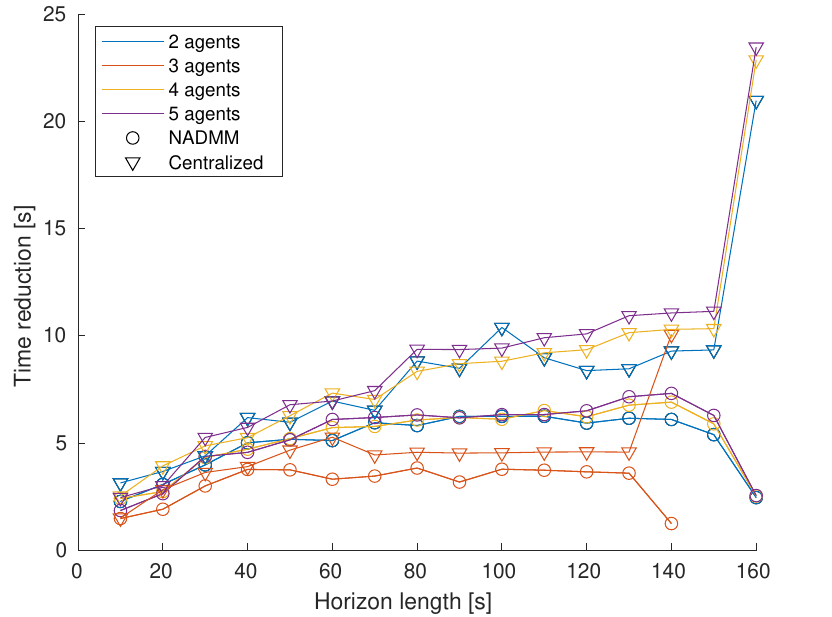} 
\caption{The improvement in mission time, i.e., the improvement in $t_f$.} 
\label{fig:time_reduction}
\end{center}
\end{figure}

\begin{figure}[tb]
\begin{center}
\includegraphics[clip, trim=4.2cm 8.5cm 3.5cm 9cm, width=0.45\textwidth]{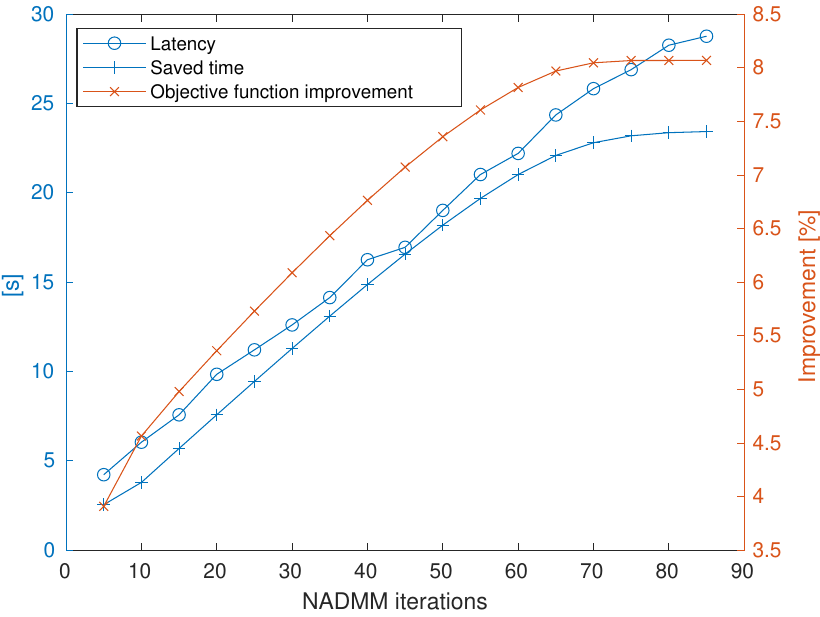} 
\caption{Objective function improvement, latency time and time improvement for varying number of NADMM iterations for the problem with five agents and horizon length 160 s.} 
\label{fig:iterations}
\end{center}
\end{figure}

\section{Conclusions}

In this work we have considered a multi-agent motion planning (MAMP) problem with the additional constraint that all agents reach their goals simultaneously. We have proposed a framework for optimized simultaneous arrival MAMP consisting of two steps. In the first step an initial solution, which is feasible and collision-free, is computed. Under the assumption that the agents are initial at rest we have proposed to use conflict-based search backwards which allows for easy padding of the trajectories to ensure simultaneous arrival. For certain choices of cost function to optimize, such as the arrival time, this provides an optimal solution. 

The second step is an improvement step where a receding-horizon optimal control problem (OCP) is posed and solved. In this step the solution found in the first step is used to warm-start the numerical solver. For the OCP solution to scale better with the number of agents we propose to solve the OCP in a distributed manner using the nonlinear alternating direction method of multipliers (NADMM). Numerical experiments show that the improvement step is able to improve the solutions found in the first step, and that the proposed solution scales better with the number of agents than solving the OCPs in a centralized manner. The results also show that for the proposed solution the performance drops when the full horizon length is used as the improvement by each call to NADMM is limited. For this reason it is better to use a shorter horizon length which allows for many calls to NADMM.

For future work, it is interesting to consider replacing the proof-of-concept implementation with a high-performance one.

\bibliographystyle{ieeetr}
\bibliography{refs_F}

\end{document}